\documentclass[12pt, a4paper]{amsart}

\usepackage[english]{babel}
\usepackage{amsmath}
\usepackage{amssymb}
\usepackage{amscd} 
\usepackage{microtype} 
\usepackage{verbatim} 
\usepackage{color}  
\usepackage{tikz-cd}\usetikzlibrary{babel} 
\usepackage{enumerate}
\usepackage{mathtools} 
\usepackage{cite}
\usepackage[initials,msc-links]{amsrefs}
\usepackage[all]{xy}

\usepackage{hyperref} 

\usepackage[OT2, T1]{fontenc}

\title[Profinite rigidity of lattices]
 {On the profinite rigidity of lattices in higher rank Lie groups}

 \author[H. Kammeyer]{Holger Kammeyer}
 \author[S. Kionke]{Steffen Kionke}
 
 \address[H. Kammeyer]{Institute for Algebra and Geometry, Karlsruhe Institute of Technology, 76131 Karlsruhe, Germany}
  \email{holger.kammeyer@kit.edu}
  
 \address[S. Kionke]{Fakult\"at f\"ur Mathematik und Informatik, FernUniversit\"at in Hagen, 58084 Hagen, Germany}
 \email{steffen.kionke@fernuni-hagen.de}
 

\theoremstyle{plain}
\newtheorem{theorem}{Theorem}
\newtheorem{lemma}[theorem]{Lemma}

\newtheorem{proposition}[theorem]{Proposition}
\theoremstyle{definition}

\newtheorem*{definition*}{Definition}

\newtheorem*{observation*}{Observation}

\numberwithin{equation}{section}
\numberwithin{theorem}{section}

\DeclareMathOperator{\Aut}{Aut}

\DeclareMathOperator{\rank}{rank}

\providecommand{\fg}{\mathfrak{g}}

\providecommand{\bbR}{\mathbb{R}}
\providecommand{\bbQ}{\mathbb{Q}}

\providecommand{\bbH}{\mathbb{H}}
\providecommand{\bbA}{\mathbb{A}}
\providecommand{\bbC}{\mathbb{C}}
\DeclareMathOperator{\SL}{SL}
\DeclareMathOperator{\SU}{SU}
\DeclareMathOperator{\PSL}{PSL}

\DeclareMathOperator{\SO}{SO}

\DeclareMathOperator{\Gal}{Gal}

\providecommand{\ignore}[1]{}
\providecommand{\alg}[1]{\mathbf{#1}}

\providecommand{\R}{\mathbb{R}}
\providecommand{\Q}{\mathbb{Q}}

\providecommand{\C}{\mathbb{C}}

\newcommand*{\arXiv}[1]{ \href{http://www.arxiv.org/abs/#1}{arXiv:\textbf{#1}}}

\begin{document}

\begin{abstract}
We investigate which higher rank simple Lie groups admit profinitely but not abstractly commensurable lattices.  We show that no such examples exist for the complex forms of type \(E_8\), \(F_4\), and \(G_2\).  In contrast, there are arbitrarily many such examples in all other higher rank Lie groups, except possibly \(\mathrm{SL}_{2n+1}(\R)\), \(\mathrm{SL}_{2n+1}(\C)\), \(\mathrm{SL}_n(\mathbb{H})\), or groups of type~\(E_6\).
\end{abstract}

\maketitle

\section{Introduction}

If two residually finite groups are commensurable, so are the profinite completions.  Thus for any class \(\mathcal{C}\) of residually finite groups, the converse \emph{rigidity} question arises: if two groups from \(\mathcal{C}\) are profinitely commensurable, are they abstractly commensurable?  It is known that arithmetic groups with the congruence subgroup property can be used to construct lattices in higher rank Lie groups which are profinitely isomorphic but not abstractly isomorphic; see \cite{Aka:profinite, Aka:arithmetic, Kammeyer:profinite-commensurability,Lubotzky:finiteness}.  Aka showed \cite{Aka:arithmetic} that the profinite isomorphism class of a higher rank lattice contains only finitely many abstract commensurability types of lattices. Here we address the question whether every simple Lie group \(G\) of higher rank actually admits profinitely isomorphic but non-commensurable lattices.
The flexibility of the familiar construction might be taken as evidence,
 that the answer should be affirmative. However, it turns out that the question is more delicate, and the answer depends on \(G\).

\begin{theorem} \label{thm:nonrigidity}
  Let \(G\) be a connected simple higher rank Lie group with finite center which is
  \begin{itemize}
  \item neither a complex form of type \(E_8\), \(F_4\), or \(G_2\),
  \item nor a real or complex form of type \(E_6\),
  \item nor locally isomorphic to \(\mathrm{SL}_{2m+1}(\R)\), \(\mathrm{SL}_{2m+1}(\C)\), or \(\mathrm{SL}_m(\mathbb{H})\).
  \end{itemize}
    Then for each \(n \ge 2\), there exist cocompact lattices \(\Gamma_1, \ldots, \Gamma_n \subset G\) which are pairwise profinitely isomorphic but pairwise not commensurable.
\end{theorem}

So in most cases rigidity fails in a strong sense but the first three groups form a surprising exception.

\begin{theorem} \label{thm:rigidity}
  Let \(G\) be a connected simple complex Lie group of type \(E_8\), \(F_4\), or \(G_2\) and let \(\Gamma_1, \Gamma_2 \subset G\) be lattices such that \(\Gamma_1\) is profinitely commensurable with \(\Gamma_2\).  Then \(\Gamma_1\) is commensurable with \(\Gamma_2\).
\end{theorem}

Going through the census, the complete list of higher rank simple Lie groups currently not covered by either Theorem~\ref{thm:nonrigidity} or Theorem~\ref{thm:rigidity} is given by \(\mathrm{SL}_{2m-1}(\R)\) and \(\mathrm{SL}_{2m-1}(\C)\), and \(\mathrm{SL}_m(\mathbb{H})\) for \(m \ge 2\), as well as all the \(E_6\)-forms: \(E_{6(6)}\), \(E_{6(2)}\), \(E_{6(-14)}\), \(E_{6(-26)}\), and \(E_6(\C)\).  If the congruence subgroup property was known to hold true for higher rank groups of type \(A_n\) and \(E_6\), all these groups would likewise fall under the conclusion of Theorem~\ref{thm:nonrigidity}. For instance, the non-isomorphic but locally isomorphic algebraic groups of types $A_n$, $D_{2n+1}$ and $E_6$ in \cite[Thm. 9.12]{Prasad-Rapinchuk:weakly-comm} would give rise to examples.  Let us stress however that our result is unconditional. Though the congruence subgroup problem is generally open in type \(A_n\) and \(D_4\), we were able to exploit the partial progress made in the literature to the extent that the type \(A_n\) groups \(\mathrm{SL}_{2m}(\R)\) and \(\mathrm{SL}_{2m}(\C)\), as well as \(\mathrm{SU}(p,q)\) for \(p,q \ge 2\), and all the type \(D_4\) groups \(\mathrm{SO}^0(6,2)\), \(\mathrm{SO}^0(5,3)\), \(\mathrm{SO}^0(4,4)\), and \(\mathrm{SO}^*(8)\) are included in Theorem~\ref{thm:nonrigidity}. 

 The lattices constructed here are intrinsically cocompact, since our approach is based on arithmetic groups defined over number fields of large degree. The construction of profinitely isomorphic \emph{non-cocompact} lattices requires different methods; e.g., similar to the construction in \cite[Thm.~9.12]{Prasad-Rapinchuk:weakly-comm}. In Proposition \ref{prop:lattices-SL} we illustrate this by showing that
  $\SL_m(\bbR)$, $\SL_m(\bbC)$ and $\SL_m(\bbH)$, where $m\geq 6$ is composite, admit profinitely isomorphic, non-commensurable, non-cocompact lattices. 

For the various types of \(G\), we employ varying methods to construct the families \(\Gamma_1, \ldots, \Gamma_n \subset G\) of profinitely isomorphic but non-commensurable lattices in Theorem~\ref{thm:nonrigidity}.  There does not seem to be a uniform approach that would work in all cases.  This is however different for the question that is more commonly addressed under the term \emph{profinite rigidity} in the literature: if two groups have isomorphic profinite completions, are they isomorphic?  Here the congruence subgroup property, whenever it is known to hold for \(G\), can be used to construct profinitely isomorphic but non-isomorphic congruence subgroups in a uniform way.

\begin{theorem} \label{thm:profinite-rigidity}
Let \(G\) be a higher rank connected simple Lie group with trivial center and not isomorphic to \(\mathrm{PSL}_m(\mathbb{H})\) or to a complex or real form of type \(E_6\).  Then there exist arbitrarily many cocompact lattices in \(G\) which are pairwise profinitely isomorphic but pairwise not isomorphic.
\end{theorem}

Again, the exceptions can likely be omitted.  They only owe to the incomplete status of the congruence subgroup problem.  We see that profinite rigidity in the usual sense also fails for \(E_8(\C)\), \(F_4(\C)\), and \(G_2(\C)\), and most likey fails for lattices in all higher rank simple Lie groups.  However, the lattices we construct in the proof of Theorem \ref{thm:profinite-rigidity} are arithmetic subgroups of different congruence levels.  While they are not isomorphic, they are commensurable which is why they should not really be considered as distinct lattices.  This is why we find it more on point to ask for profinitely commensurable lattices in \(G\) of higher rank which are not commensurable.

Let us remark that profinite rigidity questions are typically only posed for residually finite groups to avoid trivial examples like \(\widehat{\Gamma \times \Lambda} \cong \widehat{\Gamma} \times \widehat{\Lambda} \cong \widehat{\Gamma}\) for any group \(\Gamma\) and any infinite simple group \(\Lambda\). Most simple Lie groups with finite center are linear and hence lattices are residually finite.  However, there are simple Lie groups with finite center which admit lattices that are not residually finite.  An example is given by the four-fold covering of $\mathrm{Sp}_n(\bbR)$ \cite{Deligne1978}. In these cases it is nevertheless reasonable to consider the profinite completion because the kernel of the completion homomorphism $\Gamma \to \widehat{\Gamma}$ is merely a finite central subgroup.

\medskip

Profinite rigidity of groups and related problems receive considerable attention in current research activities; see \cite{Reid:profinite-rigidity} for a survey. Fundamental groups of $3$-manifolds, see \cite{BMRS} and references therein, and lattices in Lie groups \cite{KKRS,Stover:PU} are the main objects of interest.

\medskip

We briefly sketch the proofs of Theorems~\ref{thm:nonrigidity} and~\ref{thm:rigidity}.  As we just commented, the proof of Theorem~\ref{thm:nonrigidity} splits up into various cases.  At this point, we shall only present the most common line of arguments that works for most real forms \(G\).  We may assume $G = \alg{G}(\R)$ for an absolutely simple, simply connected algebraic \(\R\)-group \(\mathbf{G}\).  We construct linear algebraic groups $\alg{G}_1,\dots,\alg{G}_n$ over a suitable totally real number field $k$ such that $\alg{G}_i$ is isomorphic to $\alg{G}$ at exactly one real place of $k$ and is compact at all other real places.  The lattices arise as arithmetic subgroups $\Gamma_i \subseteq \alg{G}_i(k)$, ensuring that the congruence subgroup property holds for the groups $\alg{G}_i$ (special attention is needed in type $A_m$ and \(D_4\)).  The core of the argument is a local-global principle which \emph{almost} allows us to achieve that the groups $\alg{G}_1,\dots,\alg{G}_n$ are isomorphic at all finite places of $k$.  If the groups $\alg{G}_i$ were isomorphic at all finite places of $k$, then the congruence subgroup property implies that the groups are profinitely commensurable (by Theorem~\ref{thm:toolbox} and Lemma~\ref{lem:iso-adelic}).  If the field $k$ has no non-trivial automorphisms, then Margulis superrigidity can be used to deduce that the arithmetic lattices are not commensurable (Theorem~\ref{thm:superrigidity}).  There is however a caveat: The local-global principle only allows to control the isomorphism type except for one  finite place.  But since there exists only a finite number of possible \(\mathfrak{p}\)-adic types for the groups \(\mathbf{G}_i\), we can infer from Dirichlet's box principle that for arbitrary large \(n\), arbitrarily many of the groups $\alg{G}_i$ are isomorphic at every finite place.

To prove Theorem~\ref{thm:rigidity}, we apply Margulis arithmeticity to conclude that for \(i = 1,2\), \(\Gamma_i\) is commensurable with an arithmetic subgroup in a \(k_i\)-group \(\mathbf{G_i}\) over some number field \(k_i\) with precisely one complex place such that \(\mathbf{G_i}\) is anisotropic at all real places of \(k_i\).  The congruence subgroup property, which is known in the exceptional types under consideration, effects that \(\widehat{\Gamma_i}\) is commensurable with the finite adele points \(\mathbf{G_i}(\mathbb{A}^f_{k_i})\) of \(\mathbf{G_i}\).  Hence \(\mathbf{G_1}(\mathbb{A}^f_{k_1})\) is commensurable with \(\mathbf{G_2}(\mathbb{A}^f_{k_2})\) from which we conclude that the number fields \(k_1\) and \(k_2\) are \emph{arithmetically equivalent}, meaning they have equal Dedekind zeta function.  There exists a myriad of arithmetically equivalent number fields, also among totally real ones, which are not isomorphic.  However, a theorem due to Chinburg--Hamilton--Long--Reid says that arithmetically equivalent number fields with precisely one complex place are isomorphic~\cite{Chinburg-et-al:geodesics}*{Corollary~1.4}.  We conclude that \(\mathbf{G_2}\) is a \(k_1\)-twist of \(\mathbf{G_1}\).  These are classified by the noncommutative Galois cohomology set \(H^1(k_1, \mathbf{G_1})\) because \(\mathbf{G_1}\) has trivial center and trivial outer automorphism group.  The \emph{Hasse principle} for simply connected groups in combination with M.\,Kneser's vanishing result for Galois cohomology over \(\mathfrak{p}\)-adic fields therefore implies that \(\mathbf{G}_1\) is actually \(k_1\)-isomorphic to \(\mathbf{G}_2\), hence \(\Gamma_1\) is commensurable with \(\Gamma_2\).
   
\subsection*{Structure of the article.} We discuss some preliminaries in Section~\ref{sec:preliminaries}.  Sections~\ref{sec:proof-nonrigidity}, \ref{sec:proof-rigidity} and~\ref{sec:proof-profinite-rigidity} are dedicated to the proofs of Theorems~\ref{thm:nonrigidity}, \ref{thm:rigidity}, and~\ref{thm:profinite-rigidity}, respectively.  A short survey on the status of the congruence subgroup problem is included as an appendix.

\section{Preliminaries}\label{sec:preliminaries}

\subsection{Notation} Let $k$ be an algebraic number field. The set of places of $k$ is denoted by $V(k) = V_\infty(k) \cup V_f(k)$; it is the union of the set of archimedean places $V_\infty(k)$ and the set of finite places $V_f(k)$. The completion of $k$ at $v \in V(k)$ is denoted by $k_v$. 
The ring of adeles (respectively of finite adeles) of $k$ is $\bbA_k$ (resp.~$\bbA_k^f$).

\subsection{Number fields without automorphisms} The following result is well-known; e.g., \cite{Milne-Suh}*{Proposition 2.3}.

\begin{lemma}\label{lem:existence}
There are totally real number fields of arbitrarily large degree over \(\Q\) with trivial automorphism group.
\end{lemma}
\begin{proof}
For (arbitrarily large) \(n \ge 3\), let $K/\bbQ$ be a totally real Galois extension with Galois group $\Gal(K/\bbQ) \cong S_n$.
%
%
Let $H \leq \Gal(K/\Q)$ be a subgroup isomorphic to $S_{n-1}$. The fixed field $k = K^H$ is totally real, $[k:\bbQ] = n$ and $k$ has a trivial automorphism group. Indeed, an automorphism of $k$ extended to $K/\bbQ$ normalizes $H$ but $S_{n-1}$ is self-normalizing in $S_n$.
\end{proof}

Similarly, there is the following result for ``almost'' totally real fields.
\begin{lemma}\label{lem:existence-c}
There are number fields of arbitrarily large degree over \(\Q\) with precisely one complex place and with trivial automorphism group.
\end{lemma}

\begin{proof}
Fix a prime number \(p > 2\) and an irreducible rational polynomial \(P\) of degree \(p\) with exactly two non-real roots. 
Let \(K\) be the splitting field of~\(P\) with Galois group \(G \subseteq S_p\).  Given \(a \in K\) with \(P(a)=0\), the subgroup \(H \le G\) corresponding to \(k = \Q(a)\) has index \(p\).  So the group \(G \subseteq S_p\) contains an element of order \(p\) which must be a full \(p\)-cycle because \(p\) is prime.  Moreover, the non-real roots of~\(P\) are complex conjugates of one another, hence complex conjugation exhibits a nontrivial transposition in \(G\).  A symmetric group of prime order is generated by any full cycle and any transposition, so \(G = S_p\).  We observe (as in Lemma~\ref{lem:existence}) that the stabilizer $H$ of $a$ is actually self-normalizing.  Hence \(k = \bbQ(a)\) has trivial automorphism group.
\end{proof}

\subsection{Finite coverings of Lie groups}
The following lemma will be applied to reduce the proof to the case where the Lie group \(G\) is given by the \(\R\)- or \(\C\)-points of a linear algebraic group.
\begin{lemma}\label{lem:going-up-going-down}
Let $f \colon G_1 \to G$ and $g \colon G \to G_0$ be homomorphisms of Lie groups with finite kernels.  Assume that $G$ is linear. If \(G\) possesses $n$ cocompact lattices which are pairwise profinitely commensurable but pairwise not commensurable, then the same holds true for $G_1$ and $G_0$.
\end{lemma}
\begin{proof}
The linearity will only be used to ensure that all lattices in $G$ are residually finite.  It suffices to treat the case $n = 2$.  Let $\Gamma_1, \Gamma_2 \subseteq G$ be two non-commensurable, profinitely commensurable lattices. Then $\Delta_i= f^{-1}(\Gamma_i)$ is a cocompact lattice in $G_1$.  Let $K_i \subseteq \Delta_i$ be the kernel of the completion $\Delta_i \to \widehat{\Delta}_i$.  Since $\Gamma_i$ is residually finite, $K_i \subseteq \ker(f)$. Since $\ker(f)$ is finite, there is a finite index normal subgroup $\Delta'_i \subseteq \Delta_i$ which intersects $\ker(f)$ exactly in $K_i$ and thus the profinite completion of $\Delta'_i$ is isomorphic to the profinite completion of $f(\Delta'_i)$ which is a finite index normal subgroup of $\Gamma_i$.  Therefore $\Delta'_1$, $\Delta'_2$ are profinitely commensurable.  However, these groups are not commensurable, since every isomorphism between finite index subgroups maps the completion kernel $K_1$ to the kernel $K_2$, i.e., it induces an isomorphism of finite index subgroups of $\Gamma_1$ and $\Gamma_2$.

Since $\Gamma_1, \Gamma_2$ are residually finite and $\ker(g)$ is finite, there are finite index subgroups $\Gamma'_1, \Gamma'_2$ which do not intersect $\ker(g)$ so that they are isomorphic to lattices in $G_0$. As finite index subgroups of $\Gamma_1,\Gamma_2$ they are still profinitely commensurable but not commensurable. 
\end{proof}

\subsection{The congruence subgroup property and its consequences}

A key ingredient in the proof is the \emph{congruence subgroup property} (CSP): the statement that the kernel \(C(k, \mathbf{G})\) of the canonical homomorphism \(\widehat{\mathbf{G}(k)} \rightarrow \overline{\mathbf{G}(k)}\) from the arithmetic completion to the congruence completion of the \(k\)-rational points of certain \(k\)-groups \(\mathbf{G}\) is finite.  We will apply various special cases in which the congruence subgroup property is known to hold true.

\begin{theorem} \label{thm:known-csp}
  Let \(\mathbf{G}\) be a simply connected absolutely almost simple linear algebraic \(k\)-group which either is \(k\)-isotropic or has type
  \begin{itemize}
  \item \(B_l \ (l \ge 2)\),
  \item \(C_l \ (l \ge 2)\),
  \item \(D_l\ (l \ge 5)\),
  \item \(E_7, E_8, F_4, G_2\),
  \end{itemize}
  or let \(\mathbf{G}\) be the type \({}^2 A_{m-1}\) group \(\mathbf{G} = \mathbf{SU_m}(K,h)\) where \(h\) is a nondegenerate \(m\)-dimensional Hermitian form over a quadratic extension \(K/k\) with \(m \ge 3\).  Assume moreover that \(\sum_{v \in V_\infty(k)} \rank_{k_v} \!\mathbf{G} \ge 2\) and that \(k\) is not totally imaginary.  Then the congruence kernel \(C(k,\mathbf{G})\) has order at most two.  If \(\mathbf{G}\) is not topologically simply connected at some real place of \(k\), then \(C(k, \mathbf{G})\) is trivial.
\end{theorem}

The theorem is the essence of decades of research on the congruence subgroup problem.  References are \citelist{\cite{Platonov-Rapinchuk:algebraic-groups}*{Theorem~9.1, p.\,512, Theorem~9.5, p.\,513, Corollary~9.7, p.\,515, Theorems~9.23 and~9.24} \cite{Gille:kneser-tits} \cite{Prasad-Rapinchuk:computation}*{Main theorem}}.  A survey article providing extensive information on CSP can be found in~\cite{Prasad-Rapinchuk:developments-on-csp}.  The following result is another main tool for us.

\begin{theorem}\label{thm:toolbox}
Let $k$ be an algebraic number field and let $\alg{G}$ be a simply connected simple linear algebraic group over $k$. 
Let $\Gamma \subseteq \alg{G}(k)$ be an arithmetic subgroup.
Assume that $G_\infty = \prod_{v \in V_\infty(k)} \alg{G}(k_v)$ is not compact.
\begin{enumerate}[ (a)]
\item\label{it:lattices} $\Gamma \subseteq G_\infty$ is a lattice. If $\alg{G}(k_v)$ is compact for some $v \in V_\infty(k)$, then $\Gamma$ is cocompact.
\item\label{it:profinite-comm} Assume that $\alg{G}$ has CSP.
	Then the profinite completion $\widehat{\Gamma}$ is commensurable with the open compact subgroups of $\alg{G}(\bbA_k^f)$.
\end{enumerate}
\end{theorem}
\begin{proof}
  Part \eqref{it:lattices} is a famous result of Borel and Harish-Chandra \cite{Borel-Harish-Chandra}.  To prove part \eqref{it:profinite-comm}, note that by CSP, the kernel of the map $\varphi\colon \widehat{\Gamma} \to \alg{G}(\bbA_k^f)$ is finite.  Passing to a finite index subgroup of $\Gamma$, we can assume that $\varphi$ is injective. The strong approximation theorem holds since $\alg{G}$ is simply connected \cite[Theorem 7.12]{Platonov-Rapinchuk:algebraic-groups} and implies that the image of $\varphi$ is open (and compact since $\widehat{\Gamma}$ is compact).
\end{proof}

Finally, we will need to know that a collection of local isomorphisms of algebraic groups assembles to an adelic isomorphism.

\begin{lemma}\label{lem:iso-adelic}
Let $k$ be an algebraic number field and let $\alg{G}$, $\alg{H}$ be 
two semi-simple linear algebraic groups over $k$.
If $\alg{G}$, $\alg{H}$ are isomorphic at all finite places, i.e., $\alg{G} \times_k k_v \cong \alg{H} \times_k k_v$ for all $v \in V_f(k)$, then
\[
	\alg{G}(\bbA_k^f) \cong \alg{H}(\bbA_k^f)
\]
as topological groups.
\end{lemma}
\begin{proof}
By assumption the topological groups $\alg{G}(k_v)$ and $\alg{H}(k_v)$ are isomorphic at all finite places $v$ of $k$.
We pick models of $\alg{G}$, $\alg{H}$ over the ring of integers $\mathcal{O}_k$ of $k$. Then for all but finitely many places $v \in V_f(k)$, the compact subgroups $\alg{G}(\mathcal{O}_{k,v})$, $\alg{H}(\mathcal{O}_{k,v})$ are hyperspecial \cite[3.9.1]{Tits:Reductive} and hence isomorphic \cite[2.5]{Tits:Reductive}.
We deduce that
\begin{align*}
	\alg{G}(\bbA_k^f) &= \varinjlim_S \prod_{v \in S} \alg{G}(k_v) \times \prod_{v \not\in S} \alg{G}(\mathcal{O}_{k,v}) \\
	&\cong \varinjlim_S \prod_{v \in S} \alg{H}(k_v) \times \prod_{v \not\in S} \alg{H}(\mathcal{O}_{k,v}) \cong \alg{H}(\bbA_k^f)
\end{align*}
where the direct limit runs over all finite sets \(S\) of finite places of $k$.
\end{proof}

\subsection{Margulis superrigidity}
Margulis superrigidity will be used to show that certain arithmetic lattices are not abstractly commensurable.  
\begin{theorem}[Margulis]\label{thm:superrigidity}
Let $k_1$, $k_2$ be number fields and let $\alg{G_1},\alg{G_2}$ be simply connected, absolutely almost simple linear algebraic groups over $k_1$ and $k_2$ respectively.  Assume that $\sum_{v\in V_{\infty}(k_j)} \mathrm{rk}_{k_j}(G_j(k_{j,v})) \geq 2$ for all $j \in \{1,2\}$.  The arithmetic subgroups $\Gamma_1 \subset \alg{G_1}(k_1)$ and $\Gamma_2 \subset \alg{G_2}(k_2)$ are commensurable if and only if there is an isomorphism of fields $\sigma: k_1 \to k_2$ and a \(k_2\)-isomorphism of algebraic groups $\eta\colon {}^\sigma\alg{G_1} \to \alg{G_2}$.
\end{theorem}
\begin{proof}
After passing to finite index subgroups, we can assume that
there exists an isomorphism \(\delta \colon \Gamma_1 \xrightarrow{\cong} \Gamma_2\).  Then by Margulis superrigidity~\cite{Margulis:discrete-subgroups}*{Theorem~(C), p.\,259}, there exists \(\sigma \colon k_1 \to k_2\) and an epimorphism \(\eta \colon {}^\sigma \alg{G_1} \rightarrow \alg{G_2}\) such that $\delta$ agrees with $\eta$ on a finite index subgroup of~$\Gamma_1$. Without loss of generality $\delta = \eta|_{\Gamma_1}$.  The same argument applied to $\delta^{-1}$ implies that $\sigma$ is an isomorphism and further (since $\alg{G_2}$ is simply connected) that $\eta$ is an isomorphism.

Conversely, if $\sigma$ and $\eta$ exist, then $\Gamma_1$ and $\Gamma_2$ are (isomorphic to) arithmetic subgroups of the same algebraic group and are thus commensurable.
\end{proof}

\section{Proof of Theorem~\ref{thm:nonrigidity}} \label{sec:proof-nonrigidity}

In this section, \(k\) denotes a number field with \(n \coloneqq [k : \Q] > 2\) and with trivial automorphism group.  We will say \(k\) has type~I if it is totally real; see Lemma~\ref{lem:existence}.  In this case \(k_1, \dots, k_n\) denote the completions of \(k\) at the archimedean places.  We will say \(k\) has type~II if it has exactly one complex place; see Lemma~\ref{lem:existence-c}.  In this case \(k_1, \dots, k_{n-1}\) denote the completions at the archimedean places of \(k\) with \(k_1 \cong \C\).

\subsection{Reduction to \(G = \alg{G}(\R)\) or \(G = \alg{G}(\C)\).} As a first step of the proof of Theorem~\ref{thm:nonrigidity}, we observe that we may assume that the Lie group \(G\) is the group of real or complex points of a simply connected simple algebraic group.  To this end, let $G$ be a connected simple Lie group with finite center.  Let \(\mathfrak{g}\) be the Lie algebra of \(G\).  Since the center $Z(G)$ of $G$ is finite and the adjoint group $G/Z(G)$, being a subgroup of $\Aut(\mathfrak{g})$, is linear, we may assume by Lemma \ref{lem:going-up-going-down} that $G$ has trivial center.

If \(\mathfrak{g} \otimes_\R \C\) is simple, then the linear algebraic \(\R\)-group $\alg{G}_{\text{ad}} = \Aut_\bbR(\mathfrak{g})$ is absolutely simple and satisfies $\alg{G}_{\text{ad}}(\bbR)^0 = G$.  Let $\alg{G}$ be the simply connected covering of $\alg{G}_{\text{ad}}$.  By Lemma \ref{lem:going-up-going-down} it is sufficient to show for every \(n\) that $\alg{G}(\bbR)$ has $n$ profinitely isomorphic lattices which are not commensurable.

If \(\mathfrak{g} \otimes_\R \C\) is not simple, then \(\mathfrak{g}\) possesses a complex structure which turns it into a simple Lie algebra over \(\C\).  For any such structure, the \(\C\)-group \(\alg{G}_{\text{ad}} = \Aut_\bbC(\mathfrak{g})\) is (absolutely) simple and satisfies \(\alg{G}_{\text{ad}}(\C) = G\).  Again \(\alg{G}\) will denote the simply connected covering group of \(\alg{G}_{\text{ad}}\) and it is enough to find \(n\) non-commensurable profinitely isomorphic lattices in \(\alg{G}(\C)\).

With these remarks, we associated to each local isomorphism class of Lie groups \(G\) as in Theorem~\ref{thm:nonrigidity} a connected simply connected absolutely almost simple linear algebraic \(\mathbb{K}\)-group \(\mathbf{G}\) with \(\mathbb{K} = \R\) or \(\C\) which is unique up to \(\mathbb{K}\)-isomorphism.  Our task now is to find \(n\) profinitely commensurable lattices in \(\mathbf{G}(\mathbb{K})\) which are not commensurable.  The common intersection of the profinite completions then corresponds to non-commensurable lattices which are profinitely isomorphic, see~\cite{Ribes-Zalesskii:profinite-groups}*{Proposition~3.2.2, p.\,80} and, for instance, \cite{Kammeyer:l2-invariants}*{Proposition~6.39, p.\,159}.

\subsection{Overview of the proof.} \label{sec:overview}

Let us first assume that \(\mathbb{K} = \R\), in which case we take \(k\) of type~I.  The general idea of the proof in this case was already outlined in the introduction.  At this point we have to point out, however, that the argument only goes through provided the following two requirements are met.

{\begin{enumerate}[(i)]
    \setlength{\itemsep}{5pt}
  \item \label{item:inner-outer}  We need to assume that $\alg{G}$ and the \(\R\)-anisotropic real form $\alg{G}^{u}$ with \(\mathbf{G} \times_\R \C \cong_\C \mathbf{G}^u \times_\R \C\) are inner forms of each other. This is automatic unless the Dynkin diagram has symmetries, meaning \(\alg{G}\) has type \(A_m\), \(D_m\), or \(E_6\).  In these cases, one can read off from the Tits indices~\cite{Tits:classification}*{Table~II}, whether the condition is satisfied: In type \(A_m\),  the condition fails for \(\mathrm{SL}_{m+1}(\R)\) and \(\mathrm{SL}_{m+1}(\mathbb{H})\), while the groups \(\mathrm{SU}(r,s)\) with \(r+s = m+1\) satisfy this requirement.  In type \(D_m\), the groups \(\mathrm{SO}^*(2m) = \mathrm{SO}(m, \mathbb{H})\) are inner twists of the compact form \(\mathrm{SO}(2m)\).  For the groups \(\mathrm{SO}^0(r,s)\) with \(r+s = 2m\), the condition is satisfied if \(r\) and \(s\) are even and fails if \(r\) and \(s\) are odd.  Finally, in type \(E_6\), the condition is satisfied for \(E_{6(2)}\) and \(E_{6(-14)}\) and fails for \(E_{6(6)}\) and \(E_{6(-26)}\).

\item \label{item:csp-open} We need that \(k\)-anisotropic forms of $\alg{G}$ defined over $k$ satisfy the congruence subgroup property. This is still generally open in type $A_m$, $D_4$ and $E_6$. 
\end{enumerate}
}

Property~\eqref{item:csp-open} forces us to exclude type \(E_6\) altogether.  In type \(A_m\), however, CSP is known for special unitary groups $\alg{SU}(K/k,h)$ of hermitian forms \(h\) over a quadratic field extension \(K/k\) as we stated in Theorem~\ref{thm:known-csp}.  This allows us to prove Theorem~\ref{thm:nonrigidity} for the groups \(\mathrm{SU}(r,s)\) in Section~\ref{subsection:an}.  There we also present a workaround that allows us to include the groups \(\mathrm{SL}_{2m+1}(\mathbb{R})\) in spite of the failure of property~\eqref{item:inner-outer}.

Similarly, Kneser~\cite{Kneser:normalteiler} has shown that CSP holds for spinor groups.  Since there are no triality phenomena over \(\R\), we can use his result to cover all the groups of type $D_4$.  Note also that property~\eqref{item:inner-outer} fails in some \(D_m\) cases with \(m \ge 4\).  Therefore, we will sort out the type \(D_m\) groups with \(m \ge 4\) separately in Section~\ref{subsection:dn}.

With the special cases taken care of, we treat the remaining \(\R\)-groups of type \(B_m\), \(C_m\), \(E_7\), \(E_8\), \(F_4\), and \(G_2\) in Section~\ref{subsection:real-forms}.  For all these, the general strategy applies because \eqref{item:inner-outer} and~\eqref{item:csp-open} are satisfied.

Finally, we give the proof of Theorem~\ref{thm:nonrigidity} for \(\mathbb{K} = \C\) and \(\mathbf{G}\) of type \(B_m\), \(C_m\), \(D_m\) with \(m \ge 5\), and \(E_7\) in Section~\ref{subsection:complex-forms}.  Also in the complex case, type \(A_m\) and \(D_4\) need special attention because of the incomplete status of CSP.  The type \(D_4\) group \(\mathrm{SO}_8(\C)\) can again be covered by Kneser's result so that it was more convenient to include it in Section~\ref{subsection:dn}.  A similar trick as for \(\mathrm{SL}_{2m}(\R)\) also allows us to cover the type \(A_{2m-1}\) group \(\mathrm{SL}_{2m}(\C)\).  This argument is included in Section~\ref{subsection:an}.

\subsection{Type $A_\bullet$: $\mathrm{SU}(r,s)$, $\mathrm{SL}_{2m}(\bbR)$ and $\mathrm{SL}_{2m}(\bbC)$} \label{subsection:an}
As a first instance we consider the case when \(G\) is either of the groups $\mathrm{SU}(r,s)$ with $r,s \geq 2$, $\mathrm{SL}_{2m}(\bbR)$ and $\mathrm{SL}_{2m}(\bbC)$ with $m \geq 2$.  In these concrete examples it is instructive how local-global principles are key to our investigation.
 
It is well-known that hermitian forms for the extension $\bbC/\bbR$ are classified by dimension and signature.  We also recall that hermitian forms for a quadratic extension $E/F$ of $p$-adic fields are classified by dimension and discriminant $d \in \{\pm 1\}$ (i.e, $d=1$ exactly if the determinant lies in the image of the norm $\mathrm{N}_{E/F} \colon E^\times\to F^\times$).  

Fix a dimension $m \geq 2$. Suppose that $K= k(\sqrt{a})$ is a quadratic extension of $k$.  Let $V_{\infty}^{K}(k)$ denote the set of non-split archimedean places of $k$, i.e., the set of real places \(v\) where $a$ is negative with respect to the embedding $k \to k_v$.  We will use the following result of Landherr \cite{Landherr} on the existence of hermitian forms with prescribed local properties: There exists a $K/k$-hermitian form of dimension $m$ with
signature $(r_i,s_i)$ (with $r_i+s_i = m$) at $v_i \in V_{\infty}^{K}(k)$ and
discriminant $d_v \in \{\pm 1\}$ at $v \in V_f(k)$ exactly if
 \begin{itemize}
  \item  $d_v = 1$ for almost all $v \in V_f(k)$,
  \item $d_v = 1$ whenever $v$ splits in $K$, and
  \item $\prod_{v_i \in V_{\infty}^{\text{K}}(k)} (-1)^{s_i} \prod_{v \in V_f(k)} d_v = 1.$
 \end{itemize}
 The hermitian form is uniquely determined by this data. 

\subsubsection{$\SU(r,s)$ with $r,s \geq 2$.}  
We take $k$ of type~I and fix some $w^0 \in V_f(k)$. By weak approximation, there is an element $a \in k^\times$ which is negative at all real places and is 
a non-square in $k_{w^0}$.
Define  $K = k(\sqrt{a})$ and observe that by construction all archimedean places are non-split. In addition, the finite place $w^0$ does not split in $K$. Using the result of Landherr, there is a unique hermitian form $h_j$ such that
\begin{enumerate}[(i)]
\item the signature at the $j$-th real place is $(r,s)$,
\item the signature at all other real places is $(r+s,0)$,
\item the discriminant at $w^0$ is $d_{w^0} = (-1)^s$, and
\item the discriminant is $d_v = 1$ for every $v \in V_f(k)\setminus\{w^0\}$.
\end{enumerate} 
These hermitian forms define linear algebraic groups  \(\mathbf{SU}(h_j) \) over $k$ and
by Theorem \ref{thm:toolbox} \eqref{it:lattices} arithmetic subgroups \(\Gamma_j \subset \mathbf{SU}(h_j)(k)\) which are cocompact lattices in \(\alg{SU}(h_j)(k_j) = \SU(r,s)\).  These lattices have a congruence kernel of order at most two by Theorem~\ref{thm:known-csp} (Since the compact Lie group \(\mathbf{SU}(r+s)\) is topologically simply connected, the congruence kernel might however be nontrivial.)  In any case, by Lemma \ref{lem:iso-adelic}, the groups \(\mathbf{SU}(h_j)(\mathbb{A}_k^f)\) are all isomorphic because the hermitian forms \(h_j\) are isometric at each finite place.  It follows immediately from \ref{thm:toolbox} \eqref{it:profinite-comm} that $\Gamma_i$ and $\Gamma_j$ are pairwise profinitely commensurable. Finally, we observe that the groups $\alg{SU}(h_j)$ are not isomorphic. By construction $k$ has no non-trivial automorphisms and therefore superrigidity via Theorem \ref{thm:superrigidity} (we recall that $r,s \geq 2$) implies that  $\Gamma_i$ and $\Gamma_j$ are not commensurable for $i \neq j$.

\subsubsection{$\SL_{2m}(\bbR)$ and $\SL_{2m}(\bbC)$ for $m \geq 2$.}
Let $\mathbb{K}$ denote either $\bbR$ or $\bbC$ and let $G = \SL_{2m}(\mathbb{K})$.
If $\mathbb{K} = \bbR$ we take $k$ of type~I and otherwise of type~II. Pick a rational prime number $p$ which splits completely in $k$ and let $w_1,\dots, w_n$ denote the places of $k$ dividing $p$, i.e.,  $k_{w_i}  \cong \bbQ_p$. We fix an additional finite place $w^0 \in V_f(k)$ which is distinct from $w_1,\dots,w_n$.   By weak approximation, there is an element $a \in k^\times$ which satisfies: $a$ is negative in $k_i$ for all $i\geq 2$, represents a prescribed non-square element $x \in \bbQ^{\times}_p$ modulo squares at the places $w_1, \dots, w_n$ and is a non-square at $w^0$.
If $\mathbb{K} = \bbR$ we can arrange that, in addition, $a$ is positive in $k_1$.
Define $K = k(\sqrt{a})$. Then $V_\infty^K(k)$ contains all but the first archimedean places. By construction, the places $w_1,\dots, w_n, w^0$ are not split in $K$ and the quadratic extensions $K_{w_i}/k_{w_i}$ are isomorphic to $\bbQ_p(\sqrt{x})/\bbQ_p$.
Using the result of Landherr, there is a unique hermitian form $h_j$ such that
\begin{enumerate}[(i)]
\item the signature at $k_2,\dots,k_n$ is $(r+s,0)$, 
\item the discriminant at $w^0$ and $w_j$ equals $-1$, and
\item the discriminant is $d_v = 1$ for every $v \in V_f(k) \setminus\{w_j,w^0\}$.\end{enumerate} 
By construction, $\alg{SU}(h_j)(k_i) \cong \SU(2m)$ for all $i \geq 2$.  We observe that $\alg{SU}(h_j)(k_1) \cong \SL_{2m}(\mathbb{K})$ due to the choices of $k$ and $a$. The groups $\alg{SU}(h_j)$ are pairwise non-isomorphic, since they are non-isomorphic at one of the finite places $w_1,\dots, w_n$.  Here it is essential that $2m$ is even; only under this assumption the classification \cite[4.4]{Tits:Reductive} entails that non-isomorphic hermitian forms have non-isomorphic special unitary groups.  We claim that the topological groups $\alg{SU}(h_j)(\bbA^f_k) \cong \alg{SU}(h_i)(\bbA_k^f)$ are isomorphic for all $i, j$. Indeed,  the groups $\alg{SU}(h_j)$ are isomorphic at all finite places except for $w_1, \dots, w_n$; here a permutation of the places $w_1,\dots, w_n$ yields an isomorphism
\[
	 \prod_{\ell = 1}^n \alg{SU}(h_j)(k_{w_\ell}) \cong  \prod_{\ell = 1}^n \alg{SU}(h_i)(k_{w_\ell}).
\]
The argument used in the proof of Lemma \ref{lem:iso-adelic} implies $\alg{SU}(h_j)(\bbA^f_k) \cong \alg{SU}(h_i)(\bbA_k^f)$ and we can proceed as above to obtain the lattices $\Gamma_1, \dots, \Gamma_n$ as arithmetic subgroups of $\Gamma_j \subseteq \alg{SU}(h_j)(k)$.

\subsection{Type $D_\bullet$:  \texorpdfstring{$\SO^0(r,s)$}{SO(r,s)} with \texorpdfstring{$r+s$}{r+s} even  and \texorpdfstring{\(r+s \ge 8\)}{r+s >= 8} and \texorpdfstring{\(\SO_8(\C)\)}{SO8(C)}.} \label{subsection:dn}

\subsubsection{$\SO^0(r,s)$ with \(r+s\) even and $r+s \ge 8$.} The argument here is similar to the argument above for $\SU(r,s)$, now using the corresponding local-global principle for quadratic forms.
Recall that quadratic forms over $\bbR$ are classified by their dimension and signature; quadratic forms over $p$-adic fields
are classified by dimension, determinant (modulo squares) and the Hasse invariant (see \cite[6.\S 4]{Scharlau:quadratic-and-hermitian-forms}). 

Let \(k\) be of type~I. Take an odd prime number $p$ which completely splits in $k$ and let $w_1,\dots,w_n$ denote the finite places of $k$ dividing~$p$.
We fix an additional finite place $w^0 \in V_f(k)$  which is distinct from $w_1,\dots,w_n$.
We will use \cite[Ch.~6, Theorem 6.10]{Scharlau:quadratic-and-hermitian-forms} to construct quadratic forms $q_1,\dots,q_n$ of dimension $r+s$ over $k$  such that
\begin{itemize}
\item $q_j$ has signature $(r,s)$ at the first real place but is positive definite over all other real places,
\item $q_1, \dots, q_n$ are isometric at every finite place outside \(w_1, \dots, w_n\),
\item $q_i$ is non-split at the place $w_i$, but split at $w_j$ for all $j \neq i$.
\end{itemize}
In addition, we can achieve that $\alg{Spin}(q_i)(k_{w_i}) \cong \alg{Spin}(q_j)(k_{w_j})$ as topological groups for all $i, j$, which entails 
\[
\alg{Spin}(q_i)(\bbA_k^f) \cong \alg{Spin}(q_j)(\bbA_k^f)
\]
using the argument of Lemma \ref{lem:iso-adelic}.
By a result of Kneser, the groups $\alg{Spin}(q_j)$ have CSP; see \cite[11.1]{Kneser:normalteiler}.  As before, it follows from Theorems \ref{thm:toolbox} and \ref{thm:superrigidity} that arithmetic subgroups of the algebraic groups $\alg{Spin}(q_j)$ give rise to profinitely commensurable cocompact lattices in $\SO^0(r,s)$ which are not commensurable.

\medskip

By weak approximation, there is an element $a \in k^\times$ such that $(-1)^s a$ is positive at the first real place, $a$ is positive at all other real places and such that $(-1)^{(r+s)/2}a$ is a square in $k_{w_i}$ for all $i \in \{1,2,\dots,n\}$.  We define $q_j$ to be the unique form of determinant $a$ which has signature $(r,s)$ at the first real place, is positive definite at all other real places, has Hasse invariant $-1$ at $w_j$, and has Hasse invariant~1 at the finite places not equal to \(w^0, w_j\).  The Hasse invariant at \(w^0\) is then determined by the product formula.  It depends on \(r\), \(s\), and \(n\), but not on \(j\).

\medskip
We observe that $q_j$ is split at $w_i$ for all $i \neq j$, since \(q_j\) has $r+s$ variables, and the determinant $(-1)^{(r+s)/2}$ (i.e., $a$ modulo squares) and the Hasse invariant \(1\) are equal to the determinant and Hasse invariant of the \(\frac{r+s}{2}\)-fold orthogonal sum of hyperbolic planes \(\langle 1, -1 \rangle\).  On the other hand,  $q_j$ is non-split at $w_j$, because the Hasse invariant of $q_j$ and the split form differ.  We observe that modulo the canonical isomorphism $k_{w_i} \cong \bbQ_p \cong k_{w_j}$ the forms are isometric and hence the groups $\mathbf{Spin}(q_i)(k_{w_i})$ and $\mathbf{Spin}(q_j)(k_{w_j})$ are isomorphic as topological groups.  However, for $i\neq j$ the Witt indices of \(q_i\) and \(q_j\) at \(w_j\) differ, hence the $k_{w_j}$-rank of the groups $\mathbf{Spin}(q_i)$ and $\mathbf{Spin}(q_j)$ are different.  So these algebraic groups are not isomorphic over $k$.  Note that our discussion also covers the group \(G = \SO^*(8)\) which happens to be locally isomorphic to \(\SO^0(6,2)\).

\subsubsection{\(\SO_{8}(\C)\).}

Let \(k\) be of type~II.  For \(i = 1, \ldots, n\), there exists a non-degenerate quadratic \(k\)-form \(q_i\) of rank eight, with trivial determinant, such that \(q_i\) is anisotropic at all real places of \(k\), such that it has Hasse invariant \(-1\) at \(w_i\) and \(w^0\), and such that it has Hasse invariant \(1\) elsewhere.  Over \(\Q_p\), every nondegenerate quadratic form in at least four variables represents~1.  For \(j = 1, \ldots, n\), we thus have a decomposition \(q_i \otimes_k k_{w_j}\cong \langle 1, 1, 1, 1 \rangle \oplus h_{ij}\) for some rank four quadratic \(\Q_p\)-form \(h_{ij}\) with trivial discriminant and the same Hasse invariant as \(q_i \otimes_k k_{w_j}\).  By our choice of Hasse invariants, \(h_{ij}\) is \(\Q_p\)-anisotropic if and only if \(i = j\)~\cite{Serre:arithmetic}*{Theorem~6\,(iii), p.\,36}.  Since the form \(\langle 1, 1, 1, 1 \rangle\) is metabolic over \(\Q_p\) (we assumed \(p\) is odd), the form \(q_i\) has Witt index two at \(w_i\) and Witt index four at \(w_j\) for \(j \neq i\).  This shows that the group \(\mathbf{G}_i = \mathbf{Spin}(q_i)\) has \(\Q_p\)-rank two at \(w_i\) and \(\Q_p\)-rank four at \(w_j\) for \(j \neq i\).  The groups \(\mathbf{G}_i\) have moreover CSP by~\cite{Kneser:normalteiler}*{11.1}, so that as above we can conclude that \(\mathbf{G}(\C) \cong \mathbf{Spin_8}(\C)\) has \(n\) profinitely commensurable cocompact lattices which are not commensurable.

\subsection{The remaining real forms.} \label{subsection:real-forms}

Suppose \(G\) is locally isomorphic to $\alg{G}(\bbR)$, where $\alg{G}$ is a connected, simply connected absolutely almost simple algebraic $\bbR$-group.  Assume that $\alg{G}$ is neither of type \(A_m\), \(E_6\), nor isomorphic to \(\mathbf{Spin}(r,s)\).  Let $\alg{G}^{u}$ be the compact real form of $\alg{G}$.  Under these assumptions $\alg{G}^{u}$ is an inner form of $\alg{G}$ as we saw in Section~\ref{sec:overview}.  Let $\alg{G}_{\text{qs}}$ denote the unique quasi-split inner form of $\alg{G}$ and $\alg{G}^{u}$.

   Let \(k\) be of type~I.  By \cite[Proposition 1.10]{Borel-Harder:existence}, we can choose a quasi-split, absolutely simple, simply connected algebraic group $\alg{G}^0_{\text{qs}}$ over $k$ such that $\alg{G}^0_{\text{qs}} \times_k k_i \cong \alg{G}_{\text{qs}}$ for every $i \in \{1,2,\dots,n\}$.  Fix a nonarchimedean place \(w^0 \in V_f(k)\).  Then by \cite{Prasad-Rapinchuk:prescribed}*{Theorem~1}, for \(i = 1, \ldots, n\), there exists an inner \(k\)-twist \(\alg{G_i}\) of \(\alg{G}^0_{\text{qs}}\) such that
    \begin{itemize}
    \item \(\alg{G_i} \times_k k_i\) is isomorphic to \(\alg{G}\) while
    \item \(\alg{G_i} \times_k k_j\) is isomorphic to \(\alg{G}^{u}\) for \(j \neq i\) and
    \item \(\alg{G_i}\) is isomorphic to \(\mathbf{G}^0_{\text{qs}}\) at every finite place \(w \neq w^0\). 
    \end{itemize}

    Here, it is essential that \(\alg{G}\) and \(\alg{G}^{u}\) are inner twists of each other.  Over \(\mathfrak{p}\)-adic fields, there only exist a finite number of inner twists of a given absolutely simple group.  More precisely, in our context the non-abelian Galois cohomology \(H^1(k_{w^0}, \mathrm{Ad}(\mathbf{G_i})) \cong H^2(k_{w^0}, Z(\mathbf{G_i}))\) has cardinality at most four~\cite{Kneser:Galois}*{Satz~2 and table on p.\,254} because we assume \(\mathbf{G}\), hence \(\mathbf{G_i}\), is not of type \(A_m\).  Hence if \(n > 4(n'-1)\), the pigeon hole principle guarantees that at least \(n'\) of the groups \(\mathbf{G_1}, \ldots, \mathbf{G_n}\) are isomorphic at all finite places.  Without loss of generality, let us assume that the first \(n'\) groups \(\alg{G_1}, \ldots, \alg{G_{n'}}\) have this property; in particular, by Lemma \ref{lem:iso-adelic},
\begin{equation}\label{eq:iso-finite-adeles}
	\alg{G_1}(\bbA_k^f) \cong \alg{G_2}(\bbA_k^f) \cong \cdots \cong  \alg{G_{n'}}(\bbA_k^f).
\end{equation}

Pick arithmetic subgroups $\Gamma_i \subseteq \alg{G_i}(k)$.  The \(k\)-group \(\mathbf{G_i}\) is neither of type \(A_m\), nor \(E_6\), nor \(D_4\) (in which case it would be a spinor group) and we have \(\rank_{k_{i}} \mathbf{G_i}= \rank G \ge 2\).  Hence the congruence kernel \(C(k, \mathbf{G}_i)\) has order at most two by Theorem~\ref{thm:known-csp}.

One more time, by Theorem \ref{thm:toolbox} and Theorem~\ref{thm:superrigidity}, we obtain \(n'\) profinitely commensurable but non-commensurable cocompact lattices $\Gamma_1, \dots, \Gamma_{n'}$ in $\mathbf{G}(\R)$.

\subsection{The remaining complex forms.} \label{subsection:complex-forms}

Finally, let \(\mathbf{G}\) be a simply connected simple \(\C\)-group of type \(B_m\) (\(m \ge 2\)), \(C_m\) (\(m \ge 2\)), \(D_m\) (\(m \ge 5\)) or \(E_7\).  In the following, all uniqueness statements for algebraic groups are meant up to isomorphism over the field of definition of the group.  Let \(\mathbf{G}^u\) be the unique \(\R\)-anisotropic \(\R\)-group with \(\mathbf{G}^u \times_\R \C = \mathbf{G}\).   Let the \(\R\)-group \(\mathbf{G}_{\text{qs}}\) be the unique quasi split inner twist of \(\mathbf{G}^u\).  Finally, let \(\mathbf{G}^0\) be the unique \(\Q\)-split \(\Q\)-group with \(\mathbf{G}^0 \times_\Q \C = \mathbf{G}\).

Let \(k\) be of type~II.  Fix an odd prime \(p\) that splits in \(k\) and let \(w_1, \ldots, w_n\) be the places of \(k\) over \(p\).  By~\cite[Proposition 1.10]{Borel-Harder:existence}, we find a quasi-split, absolutely simple, simply connected algebraic group $\alg{G}^0_{\text{qs}}$ over $k$ which is isomorphic to \(\mathbf{G}_{\text{qs}}\) at all real places of \(k\) and is isomorphic to \(\mathbf{G}^0 \times \Q_p\) at \(w_1, \ldots, w_n\).  According to Kneser~\cite{Kneser:Galois}*{Satz~2 and table on p.\,254}, our assumption on the type of \(\mathbf{G}\) implies that there exists a non-trivial inner \(\Q_p\)-twist \(\mathbf{G}^p\) of \(\mathbf{G}^0 \times_\Q \Q_p\).  Fix another non-archimedian place \(w^0\) which does not lie over \(p\).  By~\cite{Prasad-Rapinchuk:prescribed}*{Theorem~1}, for each \(i = 1, \ldots, n\), there exists an inner \(k\)-twist \(\mathbf{G_i}\) of \(\mathbf{G}^0_{\text{qs}}\) such that
\begin{itemize}
\item \(\mathbf{G_i}\) is isomorphic to \(\mathbf{G}^u\) at every real place of \(k\).
\item \(\mathbf{G_i} \times_k k_{w_i}\) is isomorphic to \(\mathbf{G}^p\),
\item \(\mathbf{G_i} \times_k k_{w_j}\) is isomorphic to \(\mathbf{G}^0 \times_\Q \Q_p\) for \(j \neq i\),
\item \(\mathbf{G_i}\) is isomorphic to \(\mathbf{G}^0_{\text{qs}}\) at every finite place \(w \notin \{w^0, w_1, \ldots, w_n\}\).
\end{itemize}

By the same pigeon hole argument as above, we may assume that the first \(n'\) groups \(\mathbf{G_1}, \ldots, \mathbf{G_{n'}}\) are also isomorphic at \(w^0\).  Since \(p\) splits in \(k\), swapping any two places over \(p\) defines an automorphism of \(\mathbb{A}_k^f\).  It follows from the argument in Lemma \ref{lem:iso-adelic} that the groups \(\mathbf{G_1}(\bbA_k^f), \ldots, \mathbf{G_{n'}}(\bbA_k^f)\) are pairwise isomorphic as topological groups.  
Since CSP is known for \(\mathbf{G_i}\) by Theorem~\ref{thm:known-csp}, any arithmetic subgroups \(\Gamma_1, \ldots, \Gamma_{n'}\) of \(\mathbf{G}_1, \ldots, \mathbf{G}_{n'}\) are pairwise profinitely commensurable cocompact lattices in \(\mathbf{G}(\C)\) (Theorem \ref{thm:toolbox}) which are pairwise non-commensurable because \(k\) has no automorphism which could interchange the places \(w_1, \ldots, w_n\) (Theorem \ref{thm:superrigidity}).

\subsection{Non-cocompact lattices in special linear groups}
Our methods above exclusively produce cocompact lattices.  To complement this, we sketch a mechanism to come up with profinitely isomorphic, non-commensurable, non-cocompact lattices in special linear groups.
\begin{proposition}\label{prop:lattices-SL}
Let $G$ be either $\SL_m(\bbR)$, $\SL_m(\bbC)$, or $\SL_m(\bbH)$ 
where $m \geq 6$ is a composite number.
Then there are non-cocompact lattices $\Gamma_1, \Gamma_2 \subseteq G$ which are profinitely isomorphic but not commensurable.
\end{proposition}
\begin{proof}
Assume that $G = \SL_m(\bbR)$ where $m \geq 6$ is composite. We write
$m = dk$ with $d \geq 3$ and $k \geq 2$.
Let $C,D$ be two finite dimensional central division $\bbQ$-algebras of degree $d$, i.e., $\dim_\bbQ C = \dim_\bbQ D = d^2$. We assume that for every prime number $p$, at least one of the algebras $C_p = \bbQ_p \otimes_\bbQ C$
and $D_p = \bbQ_p \otimes_\bbQ D$ splits, this means, is isomorphic to $M_d(\bbQ_p)$. In addition, we assume that $\bbR \otimes_\bbQ C$ and $\bbR\otimes_\bbQ D$ are split. It follows from the theorem of Albert-Brauer-Hasse-Noether and the resulting description of the Brauer group $\mathrm{Br}(\bbQ)$ (see e.g. \cite[\S 18.5]{Pierce:Algebras}), that there are infinitely many such pairs of division algebras (it would be interesting to have concrete examples though). 

By construction, the central simple algebras $C \otimes_\bbQ D$ and $C\otimes_\bbQ D^{\text{op}}$ have index $d$ (see \cite[\S 18.6, Cor.]{Pierce:Algebras}), i.e.,
\[
	C \otimes_\bbQ D \cong M_d(E_1), \quad C\otimes_\bbQ D^{\text{op}} \cong M_d(E_2)
\]
for two division $\bbQ$-algebras $E_1, E_2$ of degree $d$.
We define two central simple $\bbQ$-algebras
\[
	A = M_k(E_1), \quad B = M_k(E_2).
\]
Since $d \geq 3$ is the order of $[C]$ and $[D]$ in $\mathrm{Br}(\bbQ)$ (see \cite[\S 18.6]{Pierce:Algebras}), we deduce that the algebras $C$ and $D$ are not isomorphic to their opposite algebras, that is $C \not\cong C^{\text{op}}$ and $D \not\cong D^{\text{op}}$.
It follows (using $[A] = [E_1] =[C][D]$ and $[B]= [E_2] = [C][D]^{-1}$ in $\mathrm{Br}(\bbQ)$) that $A$ is neither isomorphic to $B$ nor to $B^{\text{op}}$. Therefore, the associated reduced norm-one groups $\mathbf{G} = \SL_1(A)$ and $\mathbf{H}=\SL_1(B)$ are not isomorphic. In fact, reduced norm-one groups are isomorphic if and only if the underlying algebras are isomorphic or opposite isomorphic (see \cite[(26.9) and (26.11)]{BookOfInv}).
For every prime number $p$ we have $\mathbf{G}(\bbQ_p) \cong \mathbf{H}(\bbQ_p)$. Indeed, suppose $C$ splits at $p$, then $E_{1,p} \cong D_p$ and $E_{2,p}\cong D_p^{\text{op}}$, and consequently
$A_p = M_{k}(D_p) \cong B_p^{\text{op}}$.
If on the other hand $D$ splits at $p$, then $A_p \cong M_{k}(C_p) \cong B_p$.

Since $C$ and $D$ split over the real numbers, we have 
$\mathbf{G}(\bbR) \cong \mathbf{H}(\bbR) \cong \SL_{m}(\bbR)$. 
The groups $\mathbf{G}, \mathbf{H}$ are isotropic (because $k \geq 2$) and thus have the congruence subgroup property; see Theorem~\ref{thm:known-csp}.
As above, one can show that arithmetic subgroups of $\mathbf{G}$ and $\mathbf{H}$ are lattices in $\SL_{m}(\bbR)$ which are profinitely commensurable but not commensurable.
 
If we replace $\bbQ$ by an imaginary quadratic number field, the same construction yields profinitely isomorphic lattices in $\SL_{m}(\bbC)$. In order to obtain lattices in $\SL_m(\bbH)$, we vary the argument and write $2m = dk$ as a product of $k \geq 2$ and an even number $d \geq 4$.
We choose $C$, $D$ of degree $d$ as before, now assuming that $C$ is ramified over $\bbR$ and $D$ splits over $\bbR$.
\end{proof}

\section{Proof of Theorem~\ref{thm:rigidity}} \label{sec:proof-rigidity}

In this section we show that profinitely commensurable lattices in a connected simple complex Lie group \(G\) of type \(E_8\), \(F_4\), or \(G_2\) are abstractly commensurable.  Three features of \(G\) are used to conclude Theorem~\ref{thm:rigidity}: \(G\) is simply connected, has trivial center, and has no Dynkin diagram symmetries.  In particular, \(G\) is uniquely determined by its Lie algebra \(\mathfrak{g}\): we have \(G \cong \mathbf{G}(\C)\) for the linear algebraic \(\C\)-group \(\mathbf{G} = \Aut_\C(\mathfrak{g})\).

So let \(\Gamma_1\) and \(\Gamma_2\) be two profinitely commensurable lattices in~\(\mathbf{G}(\C)\).  We need to show that \(\Gamma_1\) is commensurable with \(\Gamma_2\).  By Margulis arithmeticity~\cite{Margulis:discrete-subgroups}*{Theorem~IX.1.11 and p.\,293/294}, for \(i=1,2\), there exists a dense number subfield \(k_i \subset \C\) whose remaining infinite places are real and there exists a simply connected absolutely almost simple \(k_i\)-group \(\mathbf{G_i}\) which is anisotropic at all real places of \(k_i\) such that for any \(k_i\)-embedding \(\mathbf{G_i} \subset \mathbf{GL_r}\) the group of \(k_i\)-integral points \(\mathbf{G_i}(\mathcal{O}_{k_i})\) is commensurable with \(\Gamma_i\).  Since commensurable groups are also profintely commensurable, it follows that \(\mathbf{G_1}(\mathcal{O}_{k_1})\) is profinitely commensurable with \(\mathbf{G_2}(\mathcal{O}_{k_2})\).  Consequently by \cite{Kammeyer:profinite-commensurability}*{Theorem~4}, \(k_1\) is arithmetically equivalent to \(k_2\).  From \cite{Chinburg-et-al:geodesics}*{Corollary~1.4}, we conclude that actually \(k_1\) is abstractly isomorphic to \(k_2\), hence the two subfields of \(\C\) are either equal or complex conjugates of one another.  Replacing \(k_2\) and \(\mathbf{G_2}\) with the complex conjugate field and group if needbe, we may assume \(k_1 = k_2 = k\).

Again owing to the exceptional type at hand, the groups \(\mathbf{G_i}\) have trivial center and no Dynkin diagram symmetries.  We thus have \(\Aut(\mathbf{G_i}) \cong \mathbf{G_i}\), so the \(k\)-isomorphism type of \(\mathbf{G_2}\) is classified by a class \(\alpha \in H^1(k, \mathbf{G_1})\) in non-commutative Galois cohomology with values in~\(\mathbf{G_1}\).  Since \(\mathbf{G_1}\) is isomorphic to \(\mathbf{G_2}\) at all infinite places of~\(k\), we see that \(\alpha\) reduces to the distinguished point in \(H^1(k_v, \mathbf{G_1})\) for all \(v \in V_\infty(k)\).  But by the main theorem of Galois cohomology of simply connected groups~\cite{Platonov-Rapinchuk:algebraic-groups}*{Theorem~6.6, p.\,289}, we have a bijection of pointed sets
\[ \theta \colon H^1(k, \mathbf{G_1}) \xrightarrow{ \ \cong \ } \prod_{v \in V_\infty(k)} H^1(k_v, \mathbf{G_1}). \]
So \(\alpha\) is the trivial twist, hence \(\mathbf{G_1}\) is \(k\)-isomorphic to \(\mathbf{G_2}\).  In particular, \(\mathbf{G_1}(\mathcal{O}_k)\) is commensurable with \(\mathbf{G_2}(\mathcal{O}_k)\), and so is \(\Gamma_1\) with \(\Gamma_2\).

\section{Proof of Theorem~\ref{thm:profinite-rigidity}} \label{sec:proof-profinite-rigidity}

Finally, in this section we give the proof of Theorem~\ref{thm:profinite-rigidity}.  Let \(\mathfrak{g}\) be the Lie algebra of \(G\).  The simple linear algebraic \(\R\)-group $\alg{G}_{\text{ad}} = \Aut_\bbR(\fg)$, satisfies $\alg{G}_{\text{ad}}(\bbR)^0 = G$.  Let $\widetilde{\alg{G}}$ be the simply connected covering of $\alg{G}_{\text{ad}}$.  We will construct profinitely isomorphic, non-isomorphic cocompact lattices $\Gamma_1,\dots, \Gamma_n \subseteq \widetilde{\alg{G}}(\bbR)$ which do not intersect the center and are thus isomorphic to lattices in $G$.
 
If $\widetilde{\alg{G}}$ is absolutely simple, then we take $k$ of type~I and set \(\alg{G} = \widetilde{\alg{G}}\).  If $\widetilde{\alg{G}}$ is not absolutely simple, then $\widetilde{\alg{G}} = \mathrm{Res}_{\bbC/\bbR}(\alg{G})$ for a simply connected simple algebraic $\bbC$-group $\alg{G}$.  In this case we choose $k$ of type~II.  Let $p$ be a rational prime number which splits completely in $k$ with finite places $w_1,\dots,w_n$ dividing $p$ and $k_{w_i} \cong \bbQ_p$ for all $i$.

By Theorem B in \cite{Borel-Harder:existence} there is a simply connected, absolutely simple algebraic $k$-group
$\alg{H}$ such that
\begin{itemize}
\item $\alg{H} \times_k k_1 \cong \alg{G}$,
\item $\alg{H}(k_j)$ is compact for all $j \geq 2$, and
\item $\alg{H}(k_{w_i}) \cong \alg{H}(k_{w_j})$ for all $i,j \in \{1,\dots,n\}$.
 \end{itemize}
If $G$ is of type $A_m$ (using $G \neq \PSL_m(\bbH)$) we can take $\alg{H}$ to be a special unitary group.  If $G$ is of type $D_4$, we can take $\alg{H}$ to be a spinor group.  By construction the group $\alg{H}$ then has the congruence subgroup property and we can find an arithmetic subgroup $\Gamma \subseteq \alg{H}(k)$ which does not intersect the congruence kernel, so that $\widehat{\Gamma} \cong \overline{\Gamma} \subseteq \alg{H}(\bbA_k^f)$.  Passing to a finite index subgroup if needbe, we may assume that $\Gamma$ intersects the center of $\mathbf{H}$ trivially, too.

We decompose the ring of finite adeles as $\bbA_k^{f} = \prod_{i=1}^n k_{w_i} \times \bbA_k^{f,p'}$ into the $p$- and $p'$-part.  We can find open compact subgroups $K_{0,i} \subseteq \alg{H}(k_{w_i})$ such that $K_{0,i} \cong K_{0,j}$ for all $i,j$ and such that 
\[
	\prod_{i=1}^n K_{0,i} \times K_f^{p'} \subseteq \overline{\Gamma}
\]
where $K_f^{p'}$ is some open compact subgroup of $\alg{H}(\bbA_k^{f,p'})$.
We choose open compact subgroups $K_{1,i} \subseteq K_{0,i}$ again such that
$K_{1,i} \cong K_{1,j}$ for all $i,j$ of sufficiently large index satisfying $|K_{0,i}:K_{1,i}| > |\alg{Z}(\alg{H})(k)|$.
Now we define the arithmetic groups
\[
	\Gamma_j = \Gamma \cap \prod_{i=1}^n K_{\delta_{i,j},i} \times K_f^{p'}
\]
where $\delta_{i,j}$ is the Kronecker-Delta. By construction $\widehat{\Gamma_j} \cong \prod_{i=1}^n K_{\delta_{i,j},i} \times K_f^{p'}$, hence the groups $\Gamma_1,\dots, \Gamma_n$ are profinitely isomorphic.  Being subgroups of $\Gamma$, they do not intersect the center.

It remains to prove that the arithmetic groups $\Gamma_1,\dots,\Gamma_n$ are pairwise non-isomorphic.  We explain this for $\Gamma_1$ and $\Gamma_2$.  Suppose for a contradiction that $\phi \colon \Gamma_1 \to \Gamma_2$ is an isomorphism.  As $\Gamma_2$ is a subgroup of $\alg{H}(k)$, it follows from Margulis superrigidity \cite[Theorem (5), p.5]{Margulis:discrete-subgroups} (using that $k$ has no automorphisms) that
there is an automorphism $\eta$ of $\alg{H}$ defined over $k$ and a homomorphism $\nu\colon \Gamma_1 \to \alg{Z}(\alg{H})(k)$ such that
\[
	\phi(\gamma) = \nu(\gamma) \eta(\gamma)
\]
for all \(\gamma \in \Gamma_1\).  Fix a Haar measure on $\alg{H}(k_{w_1})$.  Using the inclusion of $\Gamma_2$ into $\alg{H}(k_{w_1})$,  we see that the closure $\overline{\phi(\Gamma_1)}$ of $\phi(\Gamma_1)$ in $\alg{H}(k_{w_1})$ is $K_{0,1}$.  On the other hand, $\overline{\phi(\Gamma_1)}$ is contained in $\alg{Z}(\alg{H})(k)\eta(K_{1,1})$.  This can be used to derive a contradiction, since the Haar measure of the latter is strictly smaller than the Haar measure of $K_{0,1}$.  Recall that $\alg{H}(k_1)$ is unimodular \cite[I(2.2.3)]{Margulis:discrete-subgroups} and that the inner automorphism group of $\alg{H}$ has finite index in the automorphism group of $\alg{H}$.  We deduce that $\eta$ preserves the Haar measure and therefore 
\begin{align*}
\mathrm{vol}(K_{0,1})  &= \mathrm{vol}(\overline{\phi(\Gamma_1)}) \leq |\alg{Z}(\alg{H})(k)| \mathrm{vol}(\eta(K_{1,1}))\\
 &=  |\alg{Z}(\alg{H})(k)| \mathrm{vol}(K_{0,1}) |K_{0,1}:K_{1,1}|^{-1} < \mathrm{vol}(K_{0,1}) 
\end{align*}
which is a contradiction.

\subsection*{Funding}  This work was supported by the German Research Foundation, [DFG 338540207 (SPP 2026/1), DFG 441848266 (SPP 2026/2), and DFG 281869850 (RTG 2229)].

\subsection*{Acknowledgements} It is our pleasure to thank G.\,Mantilla Soler and S.\,K\"uhnlein for helpful discussions.  We are grateful for financial support from the German Research Foundation.

\end{document}